\documentclass[reqno,11pt]{amsart}

\usepackage{amssymb,amsmath}
\usepackage[all]{xy}
\usepackage{caption}
\usepackage[pdftex]{hyperref}
\usepackage{graphicx}
\usepackage{array}
\usepackage{physics}
\usepackage{color}
\usepackage{xy}
\usepackage{bm}
\usepackage{mathtools}

\usepackage {tikz}
\usetikzlibrary {automata, positioning}
\definecolor {processblue}{cmyk}{0.96,0,0,0}

\usepackage{geometry}
\allowdisplaybreaks

\newtheorem{thmIntro}{Theorem}
\newtheorem{thm}{Theorem}[section]
\newtheorem{lem}[thm]{Lemma}

\newtheorem{cor}[thm]{Corollary}
\theoremstyle{definition}
\newtheorem{defn}[thm]{Definition}
\newtheorem{ex}[thm]{Example}
\newtheorem{rem}[thm]{Remark}

\newcommand{\DD}{{\mathbb D}}

\newcommand{\NN}{{\mathbb N}}

\newcommand{\ZZ}{{\mathbb Z}}
\newcommand{\RR}{{\mathbb R}}

\newcommand{\cA}{{\mathcal{A}}}
\newcommand{\la}{\langle}
\newcommand{\ra}{\rangle}

\DeclareMathOperator{\id}{Id}
\DeclareMathOperator{\Id}{Id}

\makeatletter
\@namedef{subjclassname@2020}{%
  \textup{2020} Mathematics Subject Classification}
\makeatother

\newcommand{\eps}{\varepsilon}
\newcommand{\into}{\hookrightarrow}
\newcommand{\p}{\partial}
\newcommand{\om}{\omega}
\newcommand{\wt}{\widetilde}
\newcommand{\wh}{\widehat}
\newcommand{\R}{\mathbb{R}}
\newcommand{\Z}{\mathbb{Z}}
\newcommand{\N}{\mathbb{N}}

\newcommand{\cG}{\mathcal{G}}
\newcommand{\FF}{\mathcal{F}}
\newcommand{\std}{\mathrm{std}}

\newcommand{\Int}{\mathrm{Int}\,}
\newcommand{\Ham}{\mathrm{Ham}}
\newcommand{\Symp}{\mathrm{Symp}}
\newcommand{\Diff}{\mathrm{Diff}}

\newcommand{\abstr}{\mathrm{abstr}}

\title{Computable functions as Reeb flows}

\author{Kai Cieliebak}
\address{Department of Mathematics, Universität Augsburg, Universitätsstraße 14, 86159 Augsburg, Germany.}\email{Kai.Cieliebak@math.uni-augsburg.de }
\thanks{The three authors are supported by the Bilateral AEI–DFG project Celestial Mechanics, Hydrodynamics, and Turing Machines / Himmelsmechanik, Hydrodynamik und Turing-Maschinen (AQUACELL), with reference codes PCI2024-155042-2 and PCI2024-155062-2, and Deutsche Forschungsgemeinschaft (DFG) Projektnummer 541525489.
AGP has also been supported by projects PID2024-156578NB-I00 and PID2021-124440NB-I00, funded by MICIU/AEI/10.13039/501100011033 and by FEDER/EU. EM is supported by the Catalan Institution for Research and Advanced Studies through an ICREA Academia Prize 2021. She is also supported by the Spanish State Research Agency through the Severo Ochoa and María de Maeztu Program for Centers and Units of Excellence in R\&D, project CEX2020-001084-M, and by grant PID2023-146936NB-I00, funded by MICIU/AEI/10.13039/501100011033 and by ERDF/EU}
\author{{\'A}ngel Gonz{\'a}lez-Prieto}
\address{Department of Algebra, Geometry and Topology,
Universidad Complutense de Madrid,
Plaza Ciencias 3,
28040 Madrid, Spain \& Instituto de Ciencias Matemáticas (CSIC-UAM-UCM-UC3M), C.\ Nicolás Cabrera 13-15, 28049 Madrid, Spain.}
\email{angelgonzalezprieto@ucm.es}

\author{Eva Miranda}
\address{Laboratory de Geometria i Topologia \& SYMCREA, Departament de Matemàtiques \& IMTECH, Universitat Politècnica de Catalunya, Avinguda del Dr Marañón 44-50, 08028 Barcelona, Spain.}\email{eva.miranda@upc.edu}
\begin{document}

\begin{abstract}
We prove that, given any contact $3$-manifold and any computable
function $f: \mathbb{N} \dashrightarrow \mathbb{N}$, there exists a
defining contact form and a Poincar\'e 
section of its Reeb flow whose partially defined return map computes $f$.
\end{abstract}

\maketitle

\section{Introduction}

The realization of computation within continuous dynamical systems is
a very active research topic at the crossroads of logic, topology, and
analysis. Despite its seemingly simple nature, computation is known to
exhibit highly complex dynamics, including chaos, mixing, and
undecidability. For this reason, the ability to simulate computability
within a family of continuous dynamical systems provides strong
evidence that very intricate dynamics can arise. 

The most celebrated model of computation is given by the notion of a
\emph{Turing machine}. This is an abstraction of a computer consisting
of a simple read-write head inspecting a two-sided tape of symbols. At
each step of the computation, the head reads one symbol and, according
to the symbol read and its internal state, replaces it with another
symbol, transitions to another internal state, and shifts either to
the right or to the left along the tape in order to continue the
read-write process. This procedure continues until the machine reaches
one of its halting states, leaving the result of the computation
written on the tape. Turing machines were conceived as a way to
formalize the notion that a partially defined function $f: \N
\dashrightarrow \NN$ can be effectively computed by a mechanical
algorithm. Such a function $f$ is said to be \emph{computable} if
there exists a Turing machine $M$ such that, when we write any $n \in
\N$ on a tape in binary and run $M$ on it, then $M$ halts if and only
if $f(n)$ is defined, in which case the result of the computation is
$f(n)$ written in binary on the tape. From this perspective, $M$ is
the ``algorithm'' computing $f$, the piece of code that is executed to
calculate $f$ mechanically. 

Despite their apparent dependence on Turing machines, computable
functions are first-class citizens in computability theory. As Kleene
proved in \cite{kleene1936general}, they can be fully characterized as
partial recursive functions through a set of axioms independent of the
notion of Turing machines. In fact, in line with the
Church--Turing thesis, computable functions turn out to be models of
computation fully equivalent to Turing machines.
In this spirit,
there exist computable functions with very complex behaviour, such as
universal functions, capable of simulating any other computable
function, or undecidable functions, for which there exists no
algorithm able to determine their domain of definition. 

Partially defined functions also arise naturally in continuous
dynamical systems. Consider a manifold $M$ with a smooth vector field
$X$ and a {\em Poincar\'e section} $D\subset M$, i.e., an embedded
$2$-disk transverse to $X$. This defines a partially defined
\emph{first return map} $\Phi: D \dashrightarrow D$ as follows: if the
forward $X$-orbit of $x \in D$ hits $D$ again we set $\Phi(x)$ to be
the first return point; otherwise, $\Phi(x)$ is undefined. In
particular, if we embed $\NN \hookrightarrow D$ via the binary
expansion of an integer encoded in the square Cantor set
(see~\S\ref{ss:Cantor}), then we obtain a partially defined function
$\Phi|_\N: \N \dashrightarrow \N$.  

A particularly interesting family of dynamical systems arises from
contact geometry. Given a $3$-dimensional manifold $M$, a
\emph{contact form} is a $1$-form $\alpha \in \Omega^1(M)$ such that
$\alpha \wedge d\alpha \neq 0$. This contact form induces a maximally
non-integrable coorientable distribution $\xi = \ker \alpha$ of
planes, called the \emph{contact structure}, as well as a unique
smooth vector field $R$ characterized by $\alpha(R) = 1$ and
$d\alpha(R,\cdot) = 0$, called the \emph{Reeb field}. Reeb dynamics is
a very active area of research, with important milestones such as
Taubes' proof of the Weinstein conjecture~\cite{Taubes07}, and the
celebrated ``two or infinity'' theorem by Hofer, Wysocki, and
Zehnder~\cite{HoferWysockiZehnder1998}. 

In this setting, the main result of this note is the following.

\begin{thmIntro}\label{thm:main}
Let $(M, \xi)$ be a coorientable contact $3$-manifold. For any
computable function $f: \NN \dashrightarrow \NN$, there exists a
contact form $\beta \in \Omega^1(M)$ with $\xi = \ker \beta$ and a
$2$-dimensional disk $D \subset M$ such that the Reeb flow of
$\beta$ is transverse to $D$ and its first return map $\Phi$
satisfies $\Phi|_\N = f$. 
\end{thmIntro}

The proof of this result combines ideas from computability and contact
topology. On the computational side, it builds on the
work~\cite{gonzalez2025topological}, where models are constructed in
which computable partial functions are encoded by return maps of
volume-preserving vector fields on dynamical handlebodies. More
precisely, given a computable function $f: \N \dashrightarrow \N$,
consider a Turing machine $M_f$ computing it. One then places a $0$-cell
for each internal state of $M_f$, and connects these cells by
$1$-handles representing the transitions of $M_f$ between states,
thereby constructing a handlebody $W$. Alternatively, one may think of
$W$ as the result of ``thickening'' the graph representing $M_f$ as a
finite state machine. Then a volume preserving vector field $X$ is
defined on $W$ so that the dynamics on each $1$-handle correspond to
the read-write-shift operation performed by the Turing machine $M_f$, in
such a way that the global continuous dynamics simulate the symbolic
dynamics of $M_f$ on transverse sections of the $0$-handles.
Furthermore, by gluing the loose ends of $W$ corresponding to the
starting and halting states of $M_f$ we can arrange that $W$ has all its
$1$-handles attached on both sides to $0$-handles, and the gluing
locus $D_0$ becomes a Poincar\'e section for $X$ whose first return map
restricts to $f$ on $\N\subset D_0$.

On the geometric side, the main task is thus to realize this dynamical
handlebody as a Reeb flow. We achieve this through a series of results
that may be of independent interest. First, we construct a contact
form $\alpha$ on $W$ whose Reeb vector field is a rescaling of
$X$. This form is handcrafted in such a way that it agrees with the
standard form $-y\,dx + dt$ on a neighborhood of the
$0$-handles. Secondly, we prove that we can find an embedding $\phi: W
\hookrightarrow M$ into the given contact $3$-manifold, and a contact
form $\beta \in \Omega^1(M)$ with $\ker\beta=\xi$, such that
$\phi^*\beta = \alpha$. This implies that
$\phi(D_0) \subset M$ is a Poincar\'e section for the Reeb flow of
$\beta$ whose first return map $\Phi$ satisfies $\Phi|_\N = f$.

This note should be seen as a generalization
of~\cite{cardona2024hydrodynamic}, where Turing machines are encoded
as Euler flows through generalized shifts, but with the focus shifted
to computable functions. In \cite{cardona2024hydrodynamic}, the
authors constructed a Reeb flow on any $3$-manifold whose return map
on a Poincar\'e section represents one step in the computation of a
given Turing machine. They achieved this by embedding the contact
mapping torus of any generalized shift inside the given contact
manifold. In contrast, the construction developed in this note
realizes the whole computable function as the return map of the flow,
in such a way that the computation occurs in a single pass of the
flow. This is not merely a matter of delooping the mapping torus a
finite number of times, since the number of iterations needed to
complete the computation may depend strongly on the input. Rather,
topology plays a central role in the construction, as captured by the
dynamical handlebody. In this spirit, the proof of the contact
embedding in the present work may be seen as an extension of the
methods of \cite{cardona2024hydrodynamic} from mapping tori of
symplectomorphisms to general dynamical handlebodies.

Theorem \ref{thm:main} has a natural counterpart in symplectic
geometry. Indeed, by considering the symplectization $\widetilde{M} =
M \times \R_{>0}$ associated with $M$ and the constructed contact form
$\beta$, we can realize this Reeb flow as the restriction of a
Hamiltonian flow on $\widetilde{M}$ to an energy level. Hence, Theorem
\ref{thm:main} readily implies that one can construct Hamiltonian
flows on arbitrary $4$-dimensional symplectic manifolds realizing any
given computable function on a suitable Poincar\'e section. 
Furthermore, the correspondence between Reeb vector fields and
Beltrami fields \cite{etnyre2000contact} implies that any Reeb vector
field can be realized as a steady solution of the Euler equations for
an adapted metric. Therefore, a direct consequence of Theorem
\ref{thm:main} is that, on any given $3$-manifold, every computable
function can be realized as a Poincar\'e return map of a steady Euler
flow. This reinforces the idea that steady Euler flows may exhibit
extremely complex behaviour.

\section{Turing machines and computable functions}\label{sec:Turing}

We begin by recalling some background about Turing machines and
computable functions, see e.g.~\cite{moore2011nature}. 

\subsection{Turing machines}
Let $\cA$ be a finite set, called the {\em alphabet}. For this note,
without loss of generality, we will take the {\em binary alphabet}
$\cA = \{0,1\}$. 

A {\em Turing machine} $M=(Q,q_i,q_f,\delta)$ on $\cA$ consists of
\begin{itemize}
\item a finite set $Q$ of {\em (internal) states};
\item an {\em initial state} $q_i\in Q$ and a {\em final (or halting) state}
  $q_f\in Q$;\footnote{
The more general case of a finite set of halting
states can always be reduced to the case of one halting state,
see~\cite[\S4.3]{gonzalez2025topological}. } 
\item a {\em transition function}
$$
  \delta: Q\times\cA \to Q\times\cA\times\{+1,-1\}.
$$
\end{itemize}
Let $\Lambda\subset\cA^\Z$ be the set of finite two-sided sequences
$t=(t_i)_{i\in\Z}$, called the {\em tape states}.
Here ``finite'' means that $t_i=0$ for all but finitely many $i$.
The Turing machine induces a map
$$
  \Delta_M: Q\times\Lambda \to Q\times\Lambda
$$
on the set $Q\times\Lambda$ of {\em computation states} as follows. 
Given $(q,t)\in Q\times\Lambda$, write $\delta(q,t_0)=(q',s,\eps)$ and define
$\Delta_M(q,t):=(q',t')$ with
$$
  t_n' := \begin{cases}
    t_{n+\eps} & n \neq -\eps, \cr  
    s & n = -\eps.  
  \end{cases}
$$
The interpretation is that the read-write head (which is always at
position $0$) reads out the entry $t_0$ from $t$ and replaces it by
$s$. If $\eps=+1$ (resp.~$-1$) it moves one position to the right
(resp.~left), or equivalently, the sequence is shifted one position to
the left (resp.~right). Finally, the internal state $q$ is replaced by
$q'$. 

A Turing machine is called {\em reversible} if the map $\Delta_M$ is
injective.
By Bennett's theorem (see~\cite[\S4.1]{gonzalez2025topological}), for
every Turing machine there exists an equivalent reversible Turing machine. 

\subsection{Computable functions}\label{sec:computable-functions}

Every Turing machine $M$ computes a partially defined function
$$
  f_M:\N \dashrightarrow \N
$$
as follows. Given $n\in\N$, consider its binary expansion
$n=\sum_{i=0}^\ell a_i\,2^i$ with $a_i\in\{0,1\}$ and set the initial
tape state $t = (t_i)$ to $t_{2i} = a_i$ and $t_{2i+1} = 1$ for $0
\leq i \leq \ell$, and $t_i = 0$ otherwise. Here the $1$'s at the odd
positions are placed just as markers to distinguish the actual digits
of $n$ from the tail of blanks. For example, the number $n = 5$ would be encoded as $\cdots00\widehat{1}101110\cdots$, where the hat denotes the $0$-th position of the sequence and the omitted entries are all $0$.

Then, we start the Turing machine $M$ with $(q_i,t)$. If $M$ halts with tape state $t'$ we set
$$
    f_M(n) := \sum_{i = 0}^\infty t'_{2i}\,2^i;
$$
otherwise, if $M$ does not halt, we say that $f_M(n)$ is
undefined. Here we allow $t'$ to have nontrivial entries $t_i'$ with
$i<0$, which is necessary for reversible Turing machines.
The function $f_M$ is defined precisely on the set $D_f \subset \N$
of those $n$ for which the machine halts. 

A partially defined function $f: \N \dashrightarrow \N$ is called {\em computable} if $f=f_M$
for some Turing machine $M$. Notice that the assignment of a partially
defined function to a Turing machine is far from being injective,
since many different Turing machines may implement the same
function. It is not surjective either, as the computable functions
form a countable set, whereas $\NN^\NN$ is uncountable. An example of
a non-computable function is, for instance, the so-called busy beaver
function \cite{rado1962non}. 

\subsection{Finite state machines}

A Turing machine $M$ can be equivalently represented as a {\em finite
  state machine} $\cG_M$. This is a directed graph with vertex set $Q$
whose directed edges are labelled with triples
$$
  (\alpha,\beta,\eps) \in \cA\times\cA\times\{+1,-1\}.
$$
Here the label $(\alpha,\beta,\eps)$ of an edge from $q$ to $q'$ is
determined by
$$
  \delta(q,\alpha) = (q',\beta,\eps).
$$
Two vertices of $\cG_M$ are marked as special: the vertex
$q_f$ corresponding to the final state, which has no outgoing edges;
and the vertex $q_i$ of the initial state, which has no incoming edges.
Each vertex $q\neq q_f$ has exactly $|\cA|$ outgoing edges, with the
elements $\alpha\in\cA$ in the first entry of their labels. The number
of incoming edges at a vertex $q\neq q_i$ is unconstrained and there
can be self-loops. See Figure~\ref{fig-example-graph} for an example
of a finite state machine associated to a Turing machine.

\begin{figure}[h!]
    \centering
    \begin {tikzpicture}[-latex ,auto ,node distance =3 cm and 3cm ,on grid ,
semithick ,
state/.style ={ circle ,top color =white , bottom color = processblue!20 ,
draw,processblue , text=blue , minimum width =1 cm}]
\node[state] (C)
{$q_f$};
\node[state] (A) [above left=of C] {$q_i$};
\node[state] (B) [above right =of C] {$q$};
\path (A) edge  node[left] {$(0,0,-1)\;$} (C);
\path (A) edge [bend left =25] node[above] {$(1,0,+1)$} (B);
\path (B) edge [loop right] node[right] {$(1,0,+1)$} (B);
\path (B) edge node[right] {$\;(0,1,-1)$} (C);
\end{tikzpicture}
    \caption{Example of a finite state machine.}
    \label{fig-example-graph}
\end{figure}
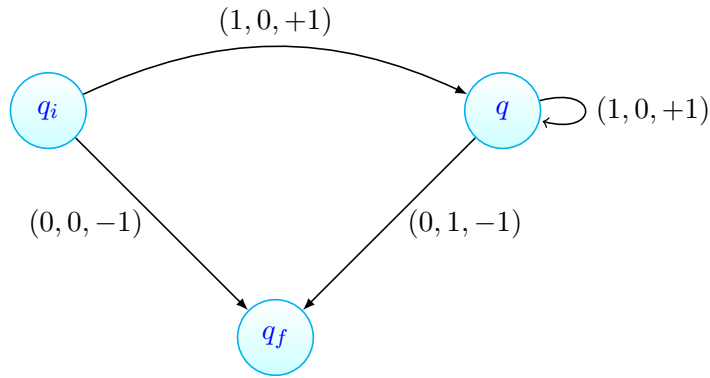

If we associate to each vertex a copy of the set $\Lambda$ of tape
states, then moving along the edges according to the map $\Delta_M:
Q\times\Lambda \to Q\times\Lambda$ defines a discrete dynamics on $\Lambda$.

The following lemma characterizes reversibility of a Turing machine in
terms of its finite state machine. It was first formulated by Morita
in \cite{morita2017reversible}, and we include its proof here for the
sake of completeness. 

\begin{lem}\label{lem:rev}
A Turing machine is reversible if and only if its finite state machine
$\cG_M$ satisfies the following condition: for any pair of incoming
edges entering the same vertex with labels $(\alpha,\beta,\eps)$ and
$(\wt\alpha,\wt\beta,\wt\eps)$, we have $\beta\neq\wt\beta$ and $\eps=\wt\eps$.
\end{lem}

\begin{proof}
Suppose first that $M$ is reversible and consider at a vertex $q'$ two 
incoming edges from vertices $q$ and $\wt q$ with labels
$(\alpha,\beta,\eps)$ and $(\wt\alpha,\wt\beta,\wt\eps)$.

First, let us prove that $\eps=\wt\eps$. Arguing by contradiction, suppose that
$\eps=+1$ and $\wt\eps=-1$ (the other case is analogous).
Consider the tape states $t,\wt t$ at $q,\wt q$ with $t_0=\alpha$,
$t_2=\wt\beta$, $\wt t_0=\wt\alpha$, $\wt t_{-2}=\beta$, and all other
entries $0$. Then $\Delta_M(q,t)=\Delta_M(\wt q,\wt t)=(q',t')$ with
$t_{-1}'=\beta$, $t_1'=\wt\beta$, and all other entries $0$. 
By reversibility of $M$ we must have $(q,t)=(\wt q,\wt t)$, which
implies $\alpha=\wt\alpha$ and therefore $\eps=\wt\eps$, contradicting
the assumption. 

Next, suppose that $\beta=\wt\beta$ (as well as $\eps=\wt\eps$).
Consider the tape states $t,\wt t$ at $q,\wt q$ with $t_0=\alpha$,
$\wt t_0=\wt\alpha$, and all other
entries $0$. Then $\Delta_M(q,t)=\Delta_M(\wt q,\wt t)=(q',t')$ with
$t_{-\eps}'=\beta$ and all other entries $0$. 
By reversibility of $M$ we must have $(q,t)=(\wt q,\wt t)$, hence
$\alpha=\wt\alpha$ and the two labelled edges are identical. 
Thus, the two labelled edges can only be different if
$\beta\neq\wt\beta$ and the property in the lemma is shown.

Conversely, suppose that $\cG_M$ satisfies the property in the lemma
and consider $(q,t)$ and $(\wt q,\wt t)$ with
$\Delta_M(q,t)=\Delta_M(\wt q,\wt t)=(q',t')$. 
This corresponds at the vertex $q'$ to two 
incoming edges from vertices $q$ and $\wt q$ with labels
$(\alpha,\beta,\eps)$ and $(\wt\alpha,\wt\beta,\wt\eps)$.
If the labelled edges were different, then the property in the
lemma would imply $\eps=\wt\eps$, hence $t_{-\eps}'=\beta=\wt\beta$,
contradicting that $\beta\neq\wt\beta$ by hypothesis.  
Thus, the labelled edges are identical, which implies $(q,t)=(\wt
q,\wt t)$ and proves reversibility. 
\end{proof}

\begin{rem}
In the case of a binary alphabet, as the one we are considering in
this note, Lemma~\ref{lem:rev} implies that each vertex has at most
two different incoming labelled edges. 
\end{rem}

\subsection{The square Cantor set}\label{ss:Cantor}

Consider the {\em ternary (or middle-thirds) Cantor set}
$$
  C := \Bigl\{\frac{2}{3}\sum_{i=0}^\infty x_i3^{-i} \;\Bigl|\; x_i\in\{0,1\}\Bigr\}
  \subset [0,1].
$$
It corresponds to the points $x\in[0,1]$ whose ternary expansion
$x=\sum_{i=1}^\infty x_i3^{-i}$ has only entries $x_i\in\{0,2\}$. 
We have a canonical embedding
$$
  \kappa_\Lambda:\Lambda\into C^2,\qquad (t_i)_{i\in\Z}\mapsto
  (x,y)=\Bigl(\frac{2}{3}\sum_{i=0}^\infty t_i3^{-i},2\sum_{i=1}^\infty t_{-i}3^{-i}\Bigr).
$$
Thus $\kappa_\Lambda$ encodes the nonnegative entries of $t$ in the
$x$-component and the negative entries in the $y$-component in the
{\em square Cantor set} $C^2\subset[0,1]^2$. In particular, encoding
$\NN$ inside $\Lambda$ as in Section \ref{sec:computable-functions},
we also get an embedding $\kappa_\N:\N\into C$. In the following we
will identify $\N$ and $\Lambda$ with their images in $C$ and $C^2$,
respectively, and drop the maps $\kappa_\N$ and $\kappa_\Lambda$.

\section{Disk maps and contact forms}

In the sequel, the closed $n$-dimensional ball in $\R^n$ of radius $r
> 0$ will be denoted by $\DD^n_r$, and its boundary by $\p
\DD^n_r$. To shorten notation, for the closed unit ball we set $\DD^n
:= \DD^n_1$. By a {\em disk} $D$ we will mean a $2$-manifold
diffeomorphic to $\DD^2$. 
We say that a function (resp.~diffeomorphism) on $D$ is {\em compactly 
supported} if it is equal to zero (resp.~the identity) on a
neighbourhood of $\p D$.
The interior of a subset $A \subset \R^n$ will be denoted by $\Int A$.

\subsection{Disk maps}

Consider a disk $D$ equipped with a symplectic form $\om$. Denote by 
$$
  \Ham_c(D)\subset \Symp_c(D) \subset \Diff_c(D)
$$
the group of compactly supported diffeomorphisms and its subgroups of
symplectic and Hamiltonian ones. Recall that the Hamiltonian vector
field $X_H$ of a function $H:D\to\R$ is defined by $i_{X_H}\om=-dH$,
and a diffeomorphism is called Hamiltonian if it is the time-$1$-map
of a time-dependent Hamiltonian vector field. 

The following lemma is well-known (see e.g.~\cite{CMPP}); we include
its simple proof for the sake of completeness. 

\begin{lem}\label{lem:disk-maps}
Let $D$ be a disk equipped with a symplectic form $\om$.
\begin{enumerate}
    \item[(a)] Given disjoint embedded closed disks $D_1,\dots,D_k\subset\Int D$ and area
  preserving embeddings $\psi_i:D_i\into\Int D$ with disjoint images,
  there exists $\psi\in\Symp_c(D)$ with $\psi|_{D_i}=\psi_i$ for $i=1,\dots,k$.
    \item[(b)] Given $\psi\in\Symp_c(D)$, there exists a smooth path
  $\psi_t\in\Symp_c(D)$ from $\psi_0=\Id$ to $\psi_1=\psi$.
    \item[(c)] Given a smooth path $(\psi_t)_{t\in[0,1]}$ in $\Symp_c(D)$ with
  $\psi_0=\id$, there exists a smooth family of compactly supported
  functions $(H_t)_{t\in[0,1]}$ such that $\dot\psi_t = X_{H_t}\circ\psi_t$.
\end{enumerate}
\end{lem}

\begin{proof}
(a) Pick $\theta\in\Diff_c(D)$ with $\theta|_{D_i}=\psi_i$ for all
  $i$. The symplectic form $\wt\om:=\theta^*\om$ agrees with $\om$
  on each $D_i$ and near $\p D$. Hence, by Moser's theorem there exists
  $\phi\in\Diff_c(D)$ with $\phi|_{D_i}=\Id$ such that
  $\phi^*\wt\om=\om$. Now $\psi:=\theta\circ\phi$ is the desired map.
  
(b) By Smale's theorem from~\cite{Smale59}, there exists a smooth path
  $\theta_t\in\Diff_c(D)$ from $\theta_0=\Id$ to $\theta_1=\psi$. Then
  $\om_t:=\theta_t^*\om$, $t\in[0,1]$, is a loop of symplectic forms
  starting and ending at $\om$. Applying Moser's theorem to a
  contraction of this loop, we find a smooth path
  $\phi_t\in\Diff_c(D)$ with $\phi_0=\phi_1=\Id$ and
  $\phi_t^*\om_t=\om$. Now $\psi_t:=\theta_t\circ\phi_t$ is the
  desired path in $\Symp_c(D)$. 

(c) Let $X_t:=\dot\psi_t\circ\psi_t^{-1}$ be the time-dependent vector
  field generating the path $\psi_t$. Then
  $0=\frac{d}{dt}\psi_t^*\om=\psi_t^*L_{X_t}\om$, and therefore $di_{X_t}\om=0$.  
  Since $i_{X_t}\om$ is compactly supported and $H^1(D,\p D)=0$, there
  exists a smooth family of compactly supported
  functions $(H_t)_{t\in[0,1]}$ such that $i_{X_t}\om=-dH_t$.
  (Explicitly, pick $z_0\in\p D$ and define
  $H_t(z):=-\int_{\gamma_z}i_{X_t}\om$ for the straight line segment
  $\gamma_z$ from $z_0$ to $z$.)
Then $X_t=X_{H_t}$ and $\dot\psi_t = X_{H_t}\circ\psi_t$.
\end{proof}

\subsection{Basic operations as disk maps}\label{ss:basic-ops}

It was observed by Moore that the operations associated to the edges
of a finite state machine $\cG_M$ can be realized by area preserving
maps of the unit disk $\DD^2\subset\R^2$ (\cite{Mo1,Mo2}, see
also~\cite[\S4.2]{gonzalez2025topological}).
To see this, consider an edge of $\cG_M$ from $q$ to $q'$ labelled by
$(\alpha,\beta,\eps)$. The associated map on $\Lambda$ is a
composition
$$
  s_{\alpha,\beta,\eps} = s_\eps\circ s_{\alpha,\beta}:\Lambda\to\Lambda.
$$
Here for $t=(t_i)_{i\in\Z}$ with $t_0=\alpha$ we have
$s_{\alpha,\beta}(t)=\wt t$ with $\wt t_0=\beta$ and $\wt t_i=t_i$ for
$i\neq 0$, and $s_+,s_-$ are the shift maps to the left resp.~right,
$$
  s_+(\dots t_{-2}t_{-1},t_0t_1t_2\dots) = (\dots t_{-2}t_{-1}t_0,t_1t_2\dots),\qquad 
  s_-(\dots t_{-2}t_{-1},t_0t_1t_2\dots) = (\dots t_{-2},t_{-1}t_0t_1t_2\dots). 
$$
These maps naturally extend to the square Cantor set $C^2$ in such a
way that they preserve the standard area form 
$\om_\std=dx\wedge dy$ on $\R^2$. For instance, the map
$s_{\alpha,\beta}$ is the identity or a translation by $\pm 2/3$ in
the $x$-direction. 
On the other hand, $s_+=s_-^{-1}$ is piecewise defined
consisting of two translations composed with the linear map
$(x,y)\mapsto(3x,y/3)$. 
More precisely, consider the left and right thirds of the square
$[0,1]^2$, 
$$
  R_0 = [0,1/3]\times [0,1],\qquad R_1=[2/3,1]\times [0,1].
$$
The affine maps $A_0,A_1:\R^2\to\R^2$ defined by
$$
  A_0(x,y) = (3x,y/3),\qquad A_1(x,y) = (3x-2,y/3+2/3)
$$
map $R_0$ and $R_1$ to the lower and upper thirds 
$$
  A_0(R_0) = [0,1]\times [0,1/3],\qquad A_1(R_1)=[0,1]\times [2/3,1],
$$
and the shift map $s_+$ is given by the restriction of $A_\alpha$ on
$C^2\cap R_\alpha$ for $\alpha=0,1$, see Figure
\ref{fig:shift-square}. Actually, $A_0$ and $A_1$ correspond to the
restrictions of the baker's map to $R_0$ and $R_1$, respectively. 

\begin{figure}[h]
    \centering
    \includegraphics[width=0.4\linewidth]{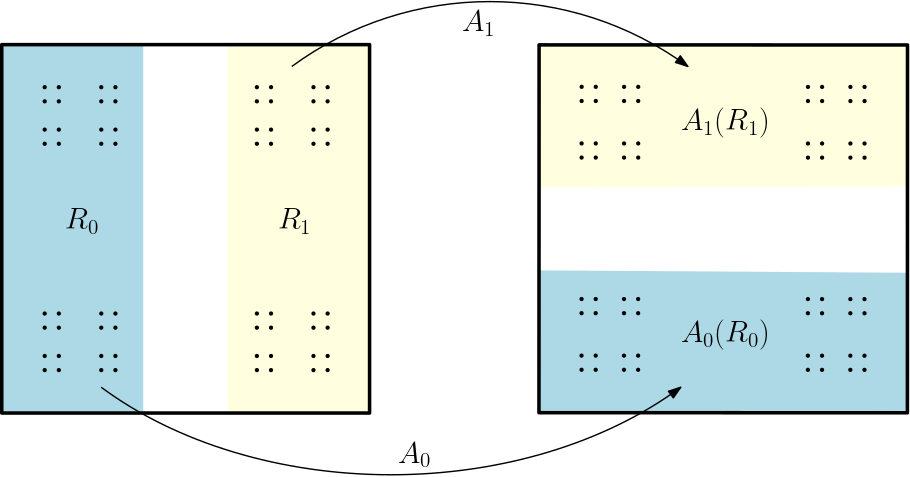}
    \caption{Regions $R_0, R_1 \subset [0,1]$ and the affine maps $A_0$ and $A_1$}
    \label{fig:shift-square}
\end{figure}

We place the unit square $[0,1]^2$ into the unit disk via the translation
$$
  T(x,y) = (x-1/2,y-1/2),
$$
so that $T([0,1]^2)=[-1/2,1/2]^2\subset\DD^2$ is centered at the origin,
and we denote $\wt R_\alpha=T(R_\alpha)$ and $\wt A_\alpha=T\circ
A_\alpha\circ T^{-1}$ for $\alpha=0,1$.  
In the sequel, we will identify the square Cantor set $C^2$ with its
image $T(C^2)\subset\DD^2$. 

Pick disjoint embedded disks $D_0,D_1\subset\Int\DD^2$ with $\wt
R_\alpha\subset\Int D_\alpha$ such that
$\wt A_0(D_0)\cap \wt A_1(D_1)=\varnothing$ and $D_1=D_0+(2/3,0)$.
By Lemma~\ref{lem:disk-maps}, there exists an area preserving diffeomorphism
$$
  \phi_+:(\DD^2,\om_\std) \stackrel{\cong}\longrightarrow (\DD^2,\om_\std)
$$
that equals the identity near $\p\DD^2$ with $\phi_+|_{D_\alpha}=\wt
A_\alpha$ for $\alpha=0,1$ (see Figure \ref{fig:shift-disk}). We set $\phi_-:=\phi_+^{-1}$ and note that
$\phi_\pm|_{C^2}=s_\pm$.  

\begin{figure}[h]
    \centering
    \includegraphics[width=0.6\linewidth]{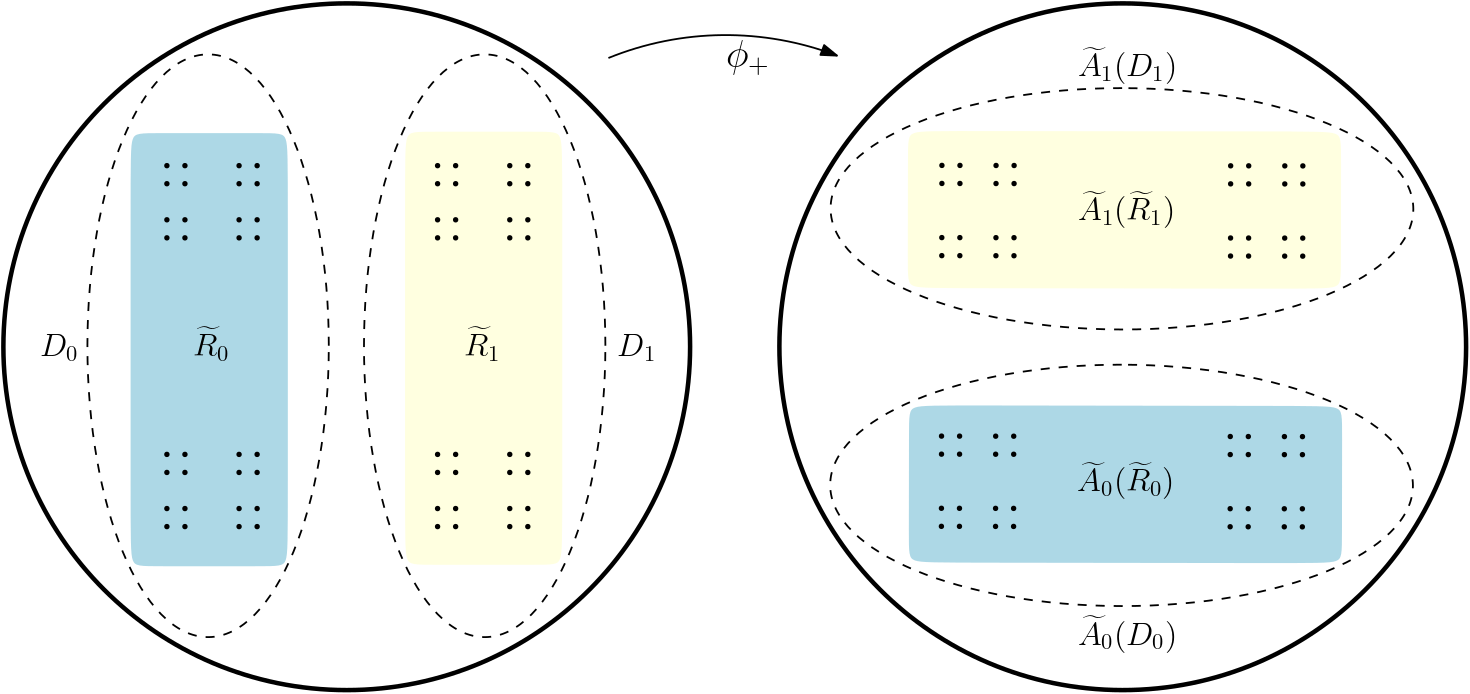}
    \caption{Effect of the diffeomorphism $\phi_+: \DD^2 \to \DD^2$}
    \label{fig:shift-disk}
\end{figure}

For $\alpha,\beta\in\{0,1\}$ we define the translations
$\phi_{\alpha,\beta}:D_\alpha\to D_\beta$ by 
$$
\phi_{\alpha,\alpha}(x,y) = (x,y), \quad \phi_{0,1}(x,y) = (x+2/3,y), \quad \phi_{1,0}(x,y) = (x-2/3,y),
$$
and the area preserving embeddings
$$
  \phi_{\alpha,\beta,\eps} :=
  \phi_\eps\circ\phi_{\alpha,\beta}:(D_\alpha,\om_\std) \into (\DD^2,\om_\std). 
$$
By construction, they satisfy
\begin{equation}\label{eq:phi-s}
  \phi_{\alpha,\beta,\eps}|_{C^2\cap D_\alpha} = s_{\alpha,\beta,\eps}.
\end{equation}

\begin{rem}\label{rem:ellipse}
A short computation shows that for $\alpha=0,1$ we can take $D_\alpha$
to be the translate centered at the center of $\wt R_\alpha$ of the
ellipse $E=\{(x,y)\mid (3x)^2+y^2\leq b^2\}$ with $1/2<b^2<8/9$.
\end{rem}

\subsection{Contact forms and structures}

In this subsection, we briefly review some basic concepts from contact
topology. For a more thorough treatment see e.g.~\cite{geiges2008introduction}.

Let $M$ be a $3$-manifold. A \emph{contact form} on $M$ is a $1$-form
$\alpha \in \Omega^1(M)$ such that $\alpha \wedge d\alpha$ is nowhere
zero. It defines a (coorientable) \emph{contact structure} $\xi = \ker
\alpha$, which is a maximally non-integrable distribution of
planes. Any contact form $\alpha'$ with $\xi = \ker \alpha'$ is said
to be a contact form defining $\xi$. 
Associated to a contact form $\alpha$ is its \emph{Reeb vector field}
$R$ defined by the properties 
$$
    \iota_{R}d\alpha = 0, \qquad \alpha(R) =1.
$$
A key property of contact structures is that they are rigid under
deformation: Gray's stability theorem asserts that, given a smooth
$1$-parameter family of contact forms $\alpha_t \in \Omega^1(M)$ for
$0 \leq t \leq 1$ with $\alpha_t=\alpha_0$ outside a compact set, there
exists an isotopy $\varphi_t: M \to M$ such that
$\varphi_t^*(\alpha_0) = \mu_t \alpha_t$, for a smooth family of
functions $\mu_t: M \to \RR_{>0}$. In particular, $(\varphi_t)_*
(\xi_0) = \xi_t$ for all $t$, where $\xi_t = \ker \alpha_t$, so the
induced contact structures are isomorphic.

A classical example of a contact form is the following.
On
$\R^2$ with coordinates $(x,y)$ we consider the standard Liouville $1$-form
$$
  \lambda_{\textrm{std}} = -y\,dx. %
$$
It induces the standard symplectic form
$\om_\std=d\lambda_{\textrm{std}} = dx\wedge dy$ on $\R^2$, and the
standard contact form and structure
$$
  \alpha_\textrm{std} = \lambda_{\textrm{std}} + dt,\qquad
  \xi_{\textrm{std}} = \ker \alpha_{\textrm{std}}
$$
on $\R^3$ with coordinates $(x,y,t)$. Since $d\alpha_{\textrm{std}} =
dx\wedge dy$, the associated Reeb vector field is $R = \partial_t$. This
example is locally canonical: by Darboux's theorem, around every point
of a $3$-manifold $M$ with contact form $\alpha$ there exists a small
neighbourhood $U \subset M$ and a diffeomorphism $\psi: \Int
\DD^3_\eps \to U$ for some $\eps>0$ such that $\psi^*\alpha =
\alpha_{\std}$.   

Note that $\alpha_\std$ is invariant under translations in the $x$-
and $t$-directions, as well as under linear maps
$(x,y,t)\mapsto(cx,y/c,t)$ with $c>0$. Rescalings $(x,y,z)\mapsto
(cx,cy,c^2t)$ with $c>0$ rescale $\alpha_\std$ by $c^2$ and thus
preserve $\xi_\std$. 
In particular, around each point of a contact $3$-manifold
  $(M,\xi)$ we find charts $\psi: \Int
\DD^3_R \to U$ such that $\psi^*\xi = \xi_{\std}$ for the contact {\em
  structure} with arbitrarily large $R$.

\subsection{Contact realization of disk maps}

Let $\om=d\lambda$ be an exact symplectic form on a disk $D$ and
consider the contact form $\alpha:=\lambda+dt$ on $D\times\R$.  

\begin{lem}\label{lem:contact-disk-maps}
Let $\psi\in\Symp_c(D)$. Then for each sufficiently large constant $\tau>0$ there
exists a smooth family of contact forms $\alpha_s$ on
$D\times[0,\tau]$ with the following properties:
\begin{itemize}
\item $\ker\alpha_s=\ker\alpha$ for all $s\in[0,1]$;
\item $\alpha_0=\alpha$, and $\alpha_s=\alpha$ near $\p(D\times[0,\tau])$ for all $s\in[0,1]$;
\item the Reeb orbit for $\alpha_1$ starting at $(z,0)$ intersects
  $D\times\{\tau\}$ at the point $(\psi(z),\tau)$.  
\end{itemize}
\end{lem}

\begin{proof}
By Lemma~\ref{lem:disk-maps}(b) and (c), there exists a smooth family
of compactly supported functions $H_t:D\to\R$, $t\in[0,1]$, whose
Hamiltonian flow $\psi_t$ satisfies $\psi_1=\psi$.   
Pick $\tau>0$ and a smooth function $\sigma:[0,\tau]\to[0,1]$ that
equals $0$ near $0$ and $1$ near $\tau$. 
The rescaled flow $\psi_{\sigma(t)}$ satisfies
$\frac{d}{dt}\psi_{\sigma(t)} = \sigma'(t)X_{H_{\sigma(t)}}$.
For $s\in[0,1]$ define the $1$-form
$$
  \alpha_s := \lambda + dt -s\sigma'(t)H_{\sigma(t)}dt
$$
on $D\times[0,\tau]$. We compute
\begin{align*}
  d\alpha_s &= \om + s\sigma'(t)dt\wedge dH_{\sigma(t)}, \cr
  \alpha_s\wedge d\alpha_s
  &= \Bigl(\bigl(1-s\sigma'(t)H_{\sigma(t)}\bigr)\om +
  s\sigma'(t)dH_{\sigma(t)}\wedge\lambda\Bigr) \wedge dt. 
\end{align*}
Thus $\alpha_s\wedge d\alpha_s>0$ for $|\sigma'|$ sufficiently small,
which can be arranged for $\tau$ large. Moreover,
$$
  i_{\p_t+\sigma'(t)X_{H_{\sigma(t)}}}d\alpha_1 =
  \sigma'(t)\Bigl(i_{X_{H_{\sigma(t)}}}\om + dH_{\sigma(t)}\Bigr) = 0,
$$
so the Reeb vector field of $\alpha_1$ is proportional to
$\p_t+\sigma'(t)X_{H_{\sigma(t)}}$ by a positive function
$D\times[0,\tau]\to\R_+$. This implies the third property, the first
two being clear. 
Finally, we use Gray's theorem to modify the $\alpha_s$ so that they
define $\xi=\ker\alpha$. 
\end{proof}

\section{Dynamical handlebodies and Reeb handlebodies}

\subsection{Dynamical handlebodies}\label{ss:dyn-handlebody}
In this section, we shall discuss the construction of a certain class
of $3$-dimensional handlebodies equipped with a $1$-dimensional
foliation, which we refer to as dynamical handlebodies. The definition
is an abstraction of the construction
in~\cite{gonzalez2025topological} to general gluing maps. 

We will write $0$-handles as $B_q=\DD^2\times \DD^1$, and $1$-handles as
$H_e=D_e\times \DD^1_\tau$ for a disk $D_e$ 
and $\tau>0$ (which may depend on $e$).
We denote coordinates on $0$- and $1$-handles by $(z,t)$ and write
$\FF_\std=\la\p_t\ra$ for the standard vertical foliation. 
We set
$$
  \p^\pm B_q = \DD^2\times\{\pm 1\},\qquad \p^\pm H_e = D_e\times\{\pm \tau\}
$$
and identify them with $\DD^2$, resp.~$D_e$, via the canonical translations. 

\begin{defn}[see Figure~\ref{fig:dyn-handlebody}]\label{def:dyn-handlebody}
A {\em dynamical handlebody} $W$ consists of the following data:
\begin{itemize}
\item a finite directed graph $G$ with vertex set $Q$ and edge set $E$
  and a preferred vertex $q_0 \in Q$;
\item a $0$-handle $B_q=\DD^2\times \DD^1$ for each vertex $q\in Q$;
\item for each edge $e\in E$ from $q$ to $q'$, a $1$-handle $H_e$
  equipped   with an oriented $1$-dimensional foliation $\FF_e$
and embeddings
$$
  \psi_e^-:\p^-H_e\into \Int\p^+B_q,\qquad
  \psi_e^+:\p^+H_e\into \Int\p^-B_{q'}
$$
\end{itemize}
satisfying the following conditions:
\begin{itemize}
\item[(i)] the images $D_e^-:=\psi_e^-(\p^-H_e)\subset\p^+B_q$ of all
  outgoing edges $e$ at $q$ are disjoint, and the images
  $D_e^+:=\psi_e^+(\p^+H_e)\subset\p^-B_{q'}$ of all incoming 
  edges $e$ at $q'$ are disjoint; 
\item[(ii)] $\FF_e=\FF_\std$ near $\p^\pm H_e$, and all leaves of $\FF_e$
  run from $\p^-H_e$ to $\p^+H_e$ with holonomy
  $h_e:\p^-H_e\to\p^+H_e$ such that
  $$
    \phi_e:=\psi_e^+\circ h_e\circ(\psi_e^-)^{-1}:D_e^-\to D_e^+
    \quad\text{satisfies}\quad \phi_e^*\om_\std=\om_\std.
  $$
\end{itemize}
An {\em isomorphism} $G:W\to\wt W$ of dynamical handlebodies with the
same underlying graph consists of diffeomorphisms
$G_e:H_e\stackrel{\cong}\longrightarrow\wt H_e$ for all edges $e$ such
that $(G_e)_*\FF_e=\wt\FF_e$, $G_e(z,t)=(g_e^\pm(z),t)$ near $\p^\pm
H_e$ for some diffeomorphisms $g_e^\pm: D_e \to \wt D_e$, and
$\psi_e^\pm=\wt\psi_e^\pm\circ g_e^\pm$.  
\end{defn}

\begin{figure}[h]
\centering
\includegraphics[width=0.7\linewidth]{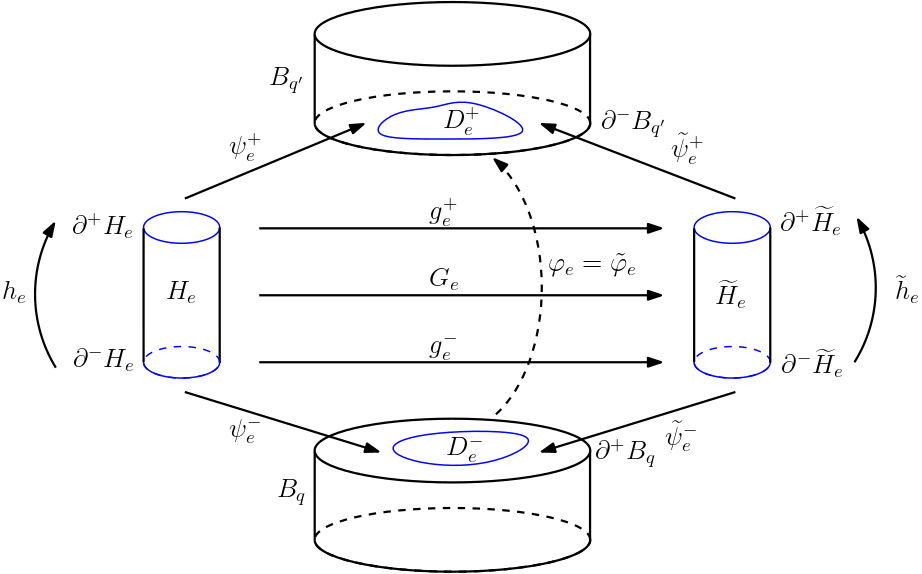}
\caption{Allowed handle attachement in a tubular handlebody}
\label{fig:dyn-handlebody}
\end{figure}

\begin{rem}\label{rem:dyn-handlebody}
Here are some comments and explanations on this definition. 

(a) The holonomy of a dynamical handlebody $W$ along $H_e$ defines an
area preserving diffeomorphism
$$
  h_e:(\p^-H_e,\om_e^-)\to(\p^+H_e,\om_e^+)\quad\text{where}\quad
  \om_e^\pm:=(\psi_e^\pm)^*\om_\std. 
$$
(b) We can view $W$ as the space obtained by gluing all $1$-handles 
$H_e$ to the adjacent $0$-handles via the maps $\psi_e^\pm$. 
It is naturally equipped with the following structures:
\begin{itemize}
\item a foliation $\FF$ given by $\FF_\std$ on $0$-handles and $\FF_e$
  on $1$-handles;
\item an $\FF$-invariant transverse area form $\omega \in \Omega^2(W)$
  that agrees with $\om_\std$ on $0$-handles;
\item a canonical cross section $D_0=\DD^2\times\{0\}\subset B_{q_0}$
  and a partially defined area preserving return map
\begin{equation}\label{eq:partial-return}
  \Phi_W:(D_0,\om_\std)\dashrightarrow (D_0,\om_\std).
\end{equation}
\end{itemize}
(c) Each dynamical handlebody is isomorphic to one for which each
$1$-handle is given by $H_e=\DD^2\times\DD^1$ with $\FF_e=\FF_\std$,
and therefore $h_e=\Id$ and $\om_e^-=\om_e^+=:\om_e$. 
However, the freedom on the $1$-handles in the definition will be
useful for the construction of dynamical handlebodies. 

(d) We can equip a dynamical handlebody $W$ with a volume form 
and a volume preserving vector field generating $\FF$. In the
situation of (c), we can take the vector field $\p_t$ and the volume
form given by $\om_\std\wedge dt$ on $0$-handles and $\om_e\wedge dt$
on $1$-handles. However, neither the
vector field nor the volume form are canonical. 
\end{rem}

Dynamical handlebodies modulo isomorphism are captured by the
following definition and lemma (which follows directly from the
definitions). 

\begin{defn}\label{def:abs-dyn-handlebody}
An {\em abstract dynamical handlebody}
consists of the following data:
\begin{itemize}
\item a finite directed graph $G$ with vertex set $Q$ and edge set $E$
  and a preferred vertex $q_0$;
\item for each edge $e\in E$ from $q$ to $q'$, embedded disks
  $D_e^\pm\subset\Int\DD^2$ and an area preserving diffeomorphism
  $$
    \phi_e:(D_e^-,\om_\std)\stackrel{\cong}\longrightarrow (D_e^+,\om_\std).
  $$
\end{itemize}
These data must satisfy the condition that, for all $q \in Q$ with
incoming edges $e_1, \ldots, e_n$ and outgoing edges $e_1', \ldots,
e_m'$, $D_{e_i}^- \cap D_{e_j}^- = \varnothing$ for all $1 \leq i, j
\leq n$ with $i \neq j$, and $D_{e_i'}^+ \cap D_{e_j'}^+ =
\varnothing$ for all $1 \leq i, j \leq m$ with $i \neq j$. 
\end{defn}

Observe that a dynamical handlebody $W$ induces a uniquely defined
abstract dynamical handlebody $W_{\abstr}$, called the
\emph{associated abstract dynamical handlebody}, with the same graph
and disks $D_e^+ = \psi_e^-(\p^-H_e)\subset\DD^2$ and $D_e^- =
\psi_e^+(\p^+H_e)\subset\DD^2$. 

\begin{lem}
Every abstract dynamical handlebody is associated to a dynamical handlebody.
Two dynamical handlebodies $W$ and $W'$ are isomorphic if and only if $W_{\abstr} = W_{\abstr}'$.
\qed
\end{lem}

\subsection{From Turing machines to dynamical handlebodies}\label{ss:Turing-dyn-handlebody}

A reversible Turing machine with finite state machine $\cG_M$ gives rise to an abstract
dynamical handlebody as follows.
Let $G$ be the graph obtained from the graph underlying $\cG_M$ by
gluing $q_i$ and $q_f$ to a single vertex $q_0$. The edges of $G$ are
the labelled edges of $\cG_M$. 
For each edge $e$ from $q$ to $q'$ labelled by $(\alpha,\beta,\eps)$,
using the notation from~\S\ref{ss:basic-ops}, we define
$$
  \phi_e:=\phi_{\alpha,\beta,\eps}=\phi_\eps\circ\phi_{\alpha,\beta}:
  (D_\alpha,\om_\std) \into (\DD^2,\om_\std),\qquad  
  D_e^-:=D_\alpha,\quad D_e^+:=\phi_e(D_e^-).
$$
It remains to verify the disjointness condition in
Definition~\ref{def:abs-dyn-handlebody}.
The part concerning the outgoing edges is clear. 
The part concerning the incoming edges follows from reversibility:
By Lemma~\ref{lem:rev}, each vertex $q'$ has at most two
different incoming edges $e$ and $\wt e$ from vertices $q$ and $\wt q$, whose labels
$(\alpha,\beta,\eps)$ and $(\wt\alpha,\wt\beta,\wt\eps)$ satisfy
$\beta\neq\wt\beta$ and $\eps=\wt\eps$.
Then the disks $\phi_{\alpha,\beta}(D_\alpha) \subset D_\beta$ and
$\phi_{\wt\alpha,\wt\beta}(D_{\wt\alpha}) \subset D_{\wt\beta}$ are disjoint, hence so are
their images $D_e^+=\phi_\eps\circ\phi_{\alpha,\beta}(D_\alpha)$ and
$D_{\wt e}^+=\phi_\eps\circ\phi_{\wt\alpha,\wt\beta}(D_{\wt\alpha})$
under the diffeomorphism $\phi_\eps$. 

Note that any dynamical handlebody $W$ realizing this abstract dynamical 
handlebody has a $0$-handle $B_q=\DD^2\times\DD^1$ for each vertex
$q\in Q\setminus\{q_i,q_f\}\cup\{q_0\}$. We think of
$B_{q_0}=\DD^2\times\DD^1$ as obtained by gluing the half-handles
$\DD^2\times[0,1]$ and $\DD^2\times[-1,0]$ associated to the vertices
$q_i$ and $q_f$. 
In view of~\eqref{eq:phi-s}, the dynamics on $W$ contains the finite
state machine $\cG_M$ as a subsystem by restricting to
$\Lambda\subset C^2\times \{0\} \subset\DD^2\times\DD^1$ in
each $0$-handle. 
In particular, the partially defined return map $\Phi_W: D_0
\dashrightarrow D_0$ of the canonical section $D_0$
in~\eqref{eq:partial-return} computes the partially defined function
$f_M$ associated to the Turing machine $M$. 

\subsection{Reeb handlebodies}\label{ss:Reeb-handlebody}

Recall that on an embedded surface $S$ in a contact $3$-manifold $(M,\xi)$, 
the intersection $TS\cap\xi$ defines a singular $1$-dimensional
distribution called the {\em characteristic distribution} (see
e.g.~\cite{geiges2008introduction}). A contact $3$-manifold is said to be
\emph{tight} if it contains no embedded disk 
whose characteristic foliation has a closed leaf~\cite[Theorem
  1.4.1]{Eliashberg92}. 
For example, the standard contact structure $\xi_\std$ on $\R^3$
defined by $\alpha_\std=-y\,dx+dt$ is tight.

\begin{ex}\label{ex:char-fol}
On the cylinder $S=\{(x,y,t)\mid x^2+y^2=r^2\}$ in $(\R^3,\xi_\std)$,
the leaves of the characteristic foliation are downward spirals given
in polar coordinates $(r,\theta)$ by $\frac{dt}{d\theta}=-r^2\sin^2\theta\leq 0$.
Hence $t(\theta+2\pi)-t(\theta)=-\pi r^2$, and any loop $t={\rm
const}$ on $S$ can be $C^\infty$-perturbed to a loop transverse to the
characteristic foliation.  
\end{ex}

For a $1$-handle $H=D\times\DD^1_\tau$, we denote its vertical boundary by
$$
  \p^vH := \p D\times\DD^1_\tau.
$$
Given a dynamical handlebody $W$, we extend the maps $\psi_e^\pm$ to 
embeddings $\Psi_e^\pm: D_e\times\R\to\DD^2\times\R$
via
$$
  \Psi_e^-(z,t):=\bigl(\psi_e^-(z),t+\tau+1\bigr),\qquad
  \Psi_e^+(z,t):=\bigl(\psi_e^+(z),t-\tau-1\bigr).
$$

\begin{defn}\label{def:Reeb-handlebody}
A {\em Reeb handlebody} $(W,\alpha)$ is a dynamical handlebody $W$
together with tight contact forms $\alpha_e$ on the $1$-handles $H_e$ satisfying 
\begin{itemize}
\item[(i)] $\alpha_e=(\Psi_e^\pm)^*\alpha_\std$ near $\p^\pm H_e$;
\item[(ii)] $\ker d\alpha_e = \FF_e$;
\item[(iii)] the characteristic foliation on $\p^vH_e$ contains
  transverse loops in the neighbourhoods of $\p^\pm H_e$ of property (i).
\end{itemize}
\end{defn}
Property (i) ensures that the $\alpha_e$ on the $1$-handles fit
together with $\alpha_\std$ on the $0$-handles to a contact form
$\alpha$ on the glued space $W$.
By property (ii), the Reeb vector field $R$ of $\alpha$ generates the 
foliation $\FF$. 
In particular, following the flow lines of $R$ we recover the
partially defined return map $\Phi_W:D_0\dashrightarrow D_0$ of $W$.
Moreover, the transverse area form $d\alpha$ agrees with the
one of the dynamical handlebody $W$. 
Note that $\alpha\wedge d\alpha$ is an $R$-invariant volume form.
Property (iii) is a technical condition included to apply Eliashberg's
uniqueness theorem for tight contact structures on the $3$-ball
agreeing on its boundary, as in the following lemma.

\begin{lem}\label{lem:contact-tubes}
Let $\xi_0,\xi_1$ be tight contact structures on a $1$-handle $H$ that
are transverse to $\p^vH$ and agree near $\p^\pm H$. Assume that the
characteristic foliation on $\p^vH$ contains transverse loops
$\gamma^\pm$ in the neighbourhoods of $\p^\pm H$ where $\xi_0=\xi_1$. 
Then there exists a diffeomorphism $\phi:H\to H$ with $\phi=\Id$ near
$\p^\pm H$ such that $\phi^*\xi_1=\xi_0$.
\end{lem}

\begin{proof}
For $i=0,1$ consider the characteristic foliation of $\xi_i$ on $\p^vH$.
It has no singularities because $\xi_i$ is transverse to $\p^vH$, and
it has no closed leaves because $\xi_i$ is tight
(see e.g.~\cite[Proposition 3.5.1]{Eliashberg92}). 
Therefore, each leaf connects the transverse loop $\gamma^-$ to
$\gamma^+$. 
It follows that there exists a diffeomorphism $\phi:H\to H$ with
$\phi=\Id$ near $\p^\pm H$ such that $\phi^*\xi_1$ and
$\xi_0$ define the same characteristic foliation on $\p H$. 
Since the characteristic foliation on $\p H$ determines a contact
structure on a neighbourhood
(see e.g.~\cite[Proposition 1.2]{Giroux91}),
we can arrange that $\phi^*\xi_1=\xi_0$ near $\p H$. 
Since $\phi^*\xi_1$ and $\xi_0$ are tight, 
Eliashberg's uniqueness theorem~\cite[Theorem 2.1.3]{Eliashberg92}
implies that by a further isotopy of $\phi$ rel $\p H$ we can arrange
$\phi^*\xi_1=\xi_0$ on $H$. 
\end{proof}

Applying this lemma to all $1$-handles, we obtain 

\begin{cor}\label{cor:contact-tubes}
If $(W,\alpha_i)$, $i=0,1$ are two Reeb handlebodies with the same
underlying dynamical handlebody $W$ and contact structures
$\xi_i=\ker\alpha_i$, then there exists a diffeomorphism $\phi:W\to W$
fixed near the $0$-handles such that $\phi^*\xi_1=\xi_0$.
\qed
\end{cor}

\begin{rem}
(a) The diffeomorphism $\phi$ is in general not an isomorphism of
dynamical handlebodies in the sense
of~\S\ref{ss:dyn-handlebody} because it need not preserve the
Reeb foliation $\FF$. 
For instance, the intersection numbers between Reeb orbits and
characteristics on $\p^vH_e$ provide obstructions to matching both the
contact structures and the Reeb foliations on a $1$-handle $H_e$. \\
(b) Property (iii) in Definition~\ref{def:Reeb-handlebody} may not be
needed to obtain the conclusion of Corollary~\ref{cor:contact-tubes}.
\end{rem}

\subsection{From dynamical to Reeb handlebodies}\label{ss:dyn-Reeb}

\begin{lem}\label{lem:Reeb-handlebody}
Let $W$ be a dynamical handlebody.
Then there exist a contact form
$\alpha$ on $W$ making $(W,\alpha)$ a Reeb handlebody.
\end{lem}

\begin{proof}
Note first that if $G:W\to\wt W$ is an isomorphism of dynamical
handlebodies and $(\wt W,\wt\alpha)$ is a Reeb handlebody, then so is
$(W,G^*\wt\alpha)$. According to Remark~\ref{rem:dyn-handlebody}(c), 
we may therefore assume that each $1$-handle has the form
$H_e=\DD^2\times\DD^1_\tau$, for any given $\tau>0$, with $\FF_e=\FF_\std$ and
$\om_e^-=\om_e^+=:\om_e$. 
We need to construct on $H_e$ a contact form $\alpha_e$ satisfying
conditions (i)--(iii) in Definition~\ref{def:Reeb-handlebody}, with
$\FF_e=\FF_\std$ and given maps $\psi_e^\pm$.  
To simplify notation, we will drop the subscript $e$.

Pick a $1$-form $\lambda$ on $\DD^2$ with $d\lambda=\om$. 
Since $(\psi^\pm)^*\om_\std=\om$, the $1$-forms 
$\lambda^\pm:=(\psi^\pm)^*\lambda_\std$ satisfy
$d\lambda^\pm=\om=d\lambda$. Since $H^1(\DD^2)=0$,
there exist smooth functions $f^\pm:\DD^2\to\R$ such that
$\lambda^\pm=\lambda+df^\pm$.  
Pick a smooth function $f:\DD^2\times\DD^1_\tau\to\R$ which
agrees with $f^\pm$ near $\DD^2\times\{\pm\tau\}$ and set 
$$
  \alpha := \lambda + df + dt = \lambda + df_t + (1+\dot f_t)dt,
$$
where we write $f_t=f(\cdot,t)$ and $\dot f_t$ denotes its derivative
in the $t$-direction. By choosing $\tau$ sufficiently large 
we can make $|\dot f_t|$ so small that $1+\dot f_t>0$. 
Then $d\alpha=\om$ and $\alpha\wedge d\alpha=(1+\dot
f_t)\om\wedge dt>0$, so $\alpha$ is a positive contact form
with Reeb vector field $(1+\dot f_t)^{-1}\p_t$. 
By construction we have $\alpha=\lambda^\pm+dt=(\Psi^\pm)^*\alpha_\std$
on neighbourhoods of $\DD^2\times\{\pm\tau\}$.
By Example~\ref{ex:char-fol}, we can make these neighbourhoods
large enough so that they contain transverse loops for the
characteristic foliation on $\p^vH=\p\DD^2\times\DD^1_\tau$.
\end{proof}

{\bf An alternative construction.}
Here we present an alternative proof of Lemma~\ref{lem:Reeb-handlebody}.
The idea is to realize the nontrivial dynamics arising from the
attaching maps within extended $0$-handles with the standard contact
structure (not contact form!), so that the extended $0$-handles are
connected by $1$-handles whose attaching maps are affine
transformations between ellipses. This will facilitate the embedding
into a given contact manifold in the following subsection. 

Consider an abstract dynamical handlebody $W_\abstr$.
For notational convenience, we will describe the construction under
the simplifying assumption that all the disks $D_e^\pm$ have the same
area $\pi\rho^2$.
It will be clear from the construction how to adjust it without this
assumption. 

Let $K^\pm$ be the maximal number of outgoing/incoming edges at
vertices of the underlying graph and set $K:=\max\{K^+,K^-\}$.
Pick $a,b>0$ with $ab=\rho^2$ and $R\geq 1$ such that the disk $\DD^2_R$
contains $K$ disjoint translates in the $x$-direction of the ellipsoid
$$
  E = \{(x,y)\mid x^2/a^2+y^2/b^2\leq 1\} \subset\DD^2_R.
$$

\begin{rem}
For the abstract dynamical handlebody associated to a Turing machine
in~\S\ref{ss:Turing-dyn-handlebody}, the assumption above is satisfied
and we have $K=2$. Moreover, according to Remark~\ref{rem:ellipse}, 
we can take $R=1$ and $a=b/3$ with $1/2<b^2<8/9$. 
\end{rem}

By the assumption above and Remark~\ref{rem:dyn-handlebody}(c), 
we can realize $W_\abstr$ by a dynamical handlebody $W$ for which each
$1$-handle has the form $H_e=\DD^2_\rho\times\DD^1$ with $\FF_e=\FF_\std$
and $\om_e^-=\om_e^+=\om_\std$.
Define the diffeomorphism
$$
  \theta: \DD^2_\rho\stackrel{\cong}\longrightarrow E,\quad (x,y)\mapsto
  (ax/\rho,by/\rho),\qquad \theta^*\lambda_\std = \lambda_\std.
$$
Consider a $0$-handle $B_q=\DD^2\times\DD^1$. Let $e_1,\dots,e_k$,
$k\leq K$ be the outgoing edges at the vertex $q$, and
$\psi_{e_i}^-:\DD_\rho\into \DD$ the corresponding attaching maps
(which by assumption have disjoint images $D_{e_i}^-$). 
Pick disjoint translates $E_i^-=E+(x_i,0)\subset\DD^2_R$, $i=1,\dots,k$ of the
ellipsoid $E$ and define
$$
  \theta_i^-: \DD^2_\rho\stackrel{\cong}\longrightarrow E_i^-,\quad (x,y)\mapsto
  \theta(x,y)+(x_i,0),\qquad (\theta_i^-)^*\lambda_\std = \lambda_\std.
$$
By Lemma~\ref{lem:disk-maps}(a), there exists $\psi_q^-\in\Symp_c(\DD^2_R)$ with
$$
  \psi_q^-|_{E_i^-}=\psi_{e_i}^-\circ(\theta_i^-)^{-1}:E_i^-\stackrel{\cong}\longrightarrow
  D_{e_i}^-\subset\DD^2_R,\qquad i=1,\dots,k.
$$
See Figure \ref{fig:gluing-alternative} for a graphical representation of these maps.

\begin{figure}[h]
\centering
\includegraphics[width=0.6\linewidth]{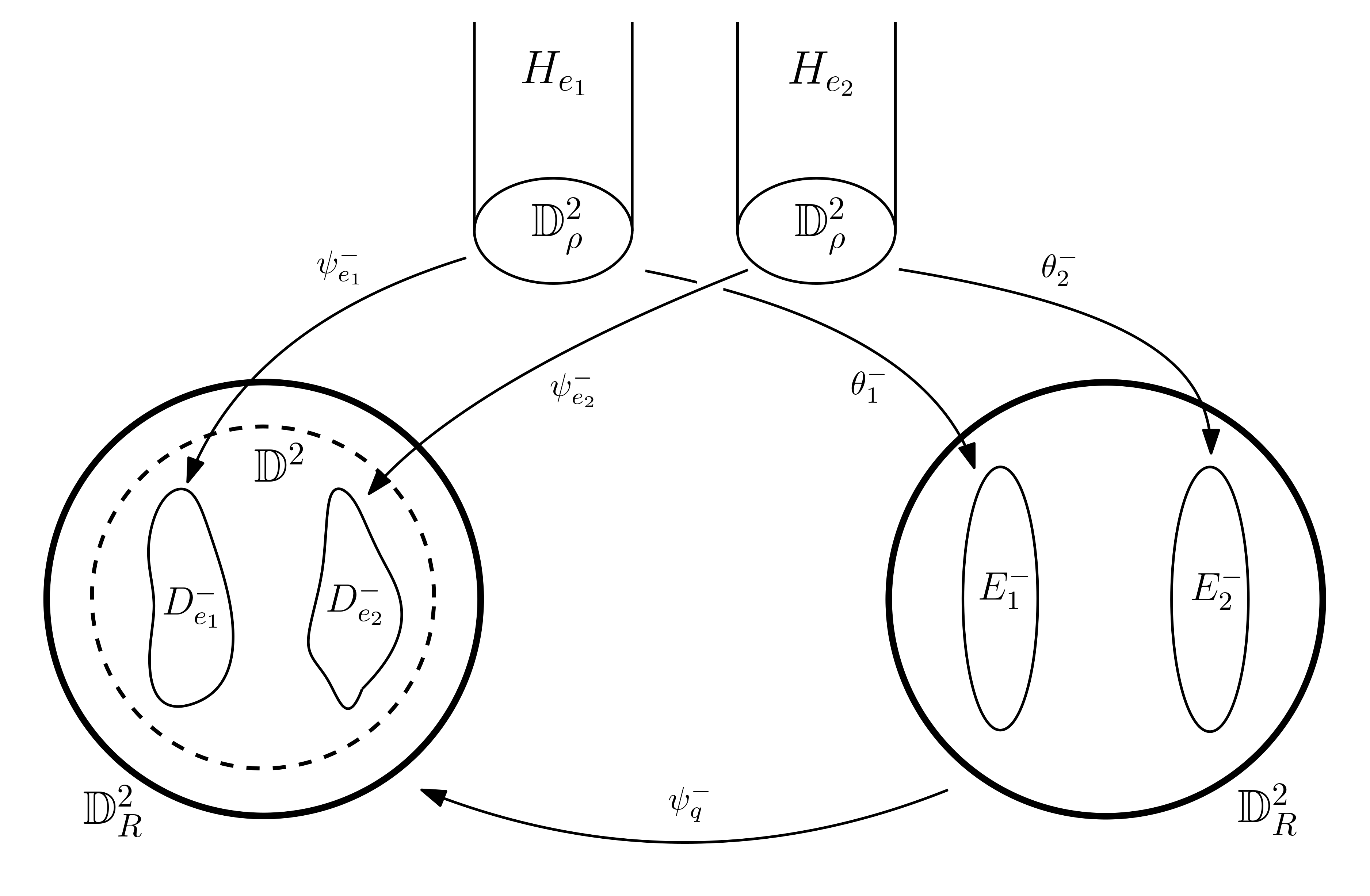}
\caption{Gluing of the $1$-handles in the alternative construction}
\label{fig:gluing-alternative}
\end{figure}

By Lemma~\ref{lem:contact-disk-maps} (applied to $\alpha_\std$ and
setting $\alpha_q:=\alpha_1$), for
$\tau>1$ sufficiently large there exists a contact form $\alpha_q$ on
$\DD^2_R\times[-1,\tau]$ with the following properties (see Figure \ref{fig:Reeb-alternative-construction}):
\begin{itemize}
\item $\ker\alpha_q=\xi_\std$;
\item $\alpha_q=\alpha_\std$ near
  $\DD^2_R\times[-1,1]\cup\p\DD^2_R\times[-1,\tau]$;
\item the Reeb orbit for $\alpha_q$ starting at $(z,1)$ intersects
  $\DD^2_R\times\{\tau\}$ at the point $\bigl((\psi_q^-)^{-1}(z),\tau\bigr)$.
\end{itemize}
We perform the analogous construction for the incoming edges at $q$
to extend $\alpha_q$ to a contact form, still denoted $\alpha_q$, on
$\wh B_q:=\DD^2_R\times[-\tau,\tau]$. We glue the $1$-handles
$H_{e_i}=\DD^2_\rho\times\DD^1$ of the outgoing edges to $\wh 
B_q$ by the maps $\theta_i^-$, and similarly for the incoming edges.
We perform this for all $0$-handles and denote the resulting space by
$\wh W$. It carries a contact form $\wh\alpha$ given by $\alpha_q$ on
each extended $0$-handle $\wh B_q$ and by $\alpha_\std$ on each $1$-handle. 

\begin{figure}[h]
\centering
\includegraphics[width=0.7\linewidth]{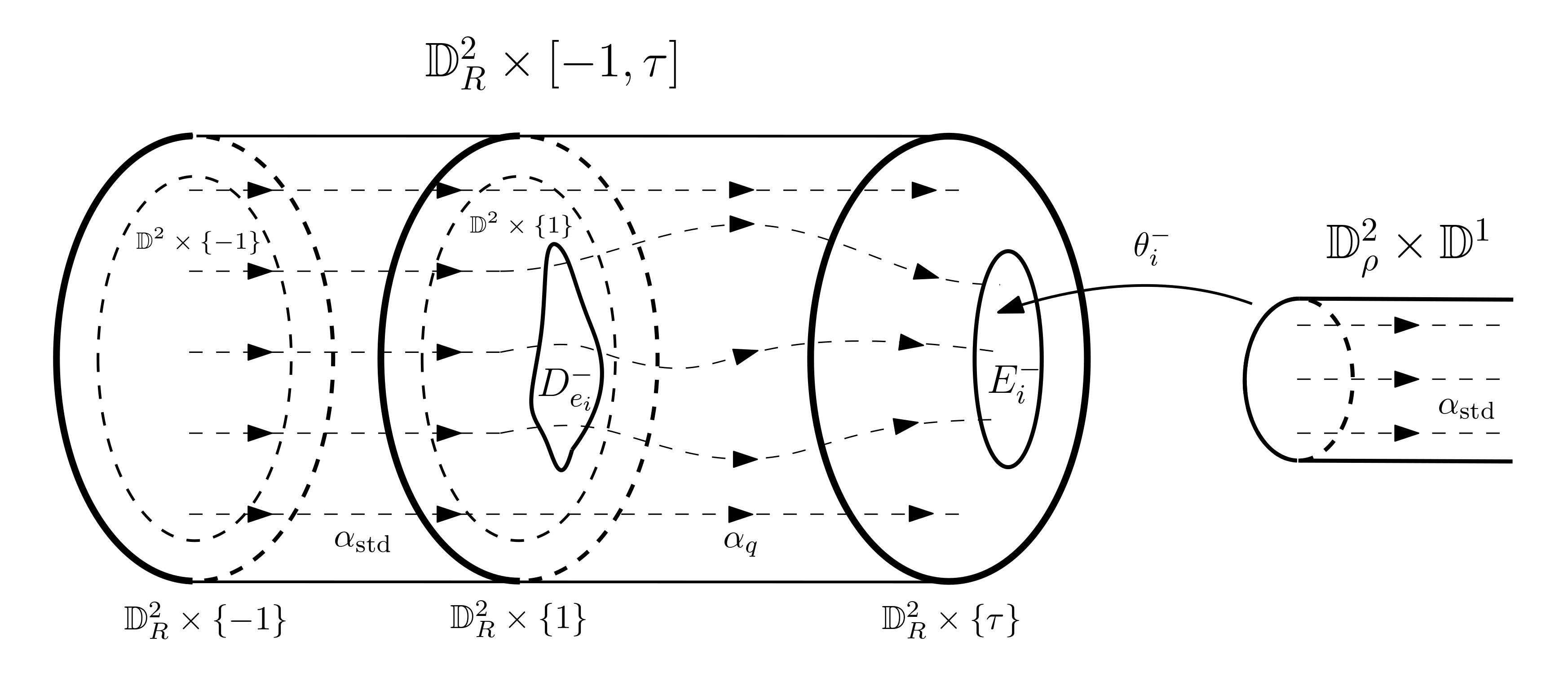}
\caption{Reeb flow in the alternative construction}
\label{fig:Reeb-alternative-construction}
\end{figure}

Note that $(\wh W,\wh\alpha)$ is not a contact handlebody in our sense
because the contact form is not $\alpha_\std$ on the extended $0$-handles $\wh B_q$.
However, it contains a contact handlebody $(\wt W,\wt\alpha)$ defined
as follows. The $0$-handles are $\wt B_q=\DD^2\times[-1,1]$ with the
contact form $\alpha_\std$. For an edge $e$ from $q$ to $q'$, the
$1$-handle $\wt H_e=\DD^2_\rho\times\DD^1_\tau$ is obtained from $\DD^2_\rho\times\DD^1$ by
gluing its negative end via the corresponding map $\theta_i^-$ to the cylinder in
$\DD^2_R\times[1,\tau]\subset\wh B_q$ swept out by $\theta_i^-(\DD^2_\rho)=E_i^-$
under the backward Reeb flow of $\alpha_q$, and similarly at the
positive end (see Figure \ref{fig:subhandlebody}). The $1$-handles $\wt H_e$ are equipped with the
restriction of the contact form $\wh\alpha$, and they are attached to
the $0$-handles by the obvious inclusions. 
By Example~\ref{ex:char-fol}, for $\tau$ large we can arrange condition (iii) of Definition \ref{def:Reeb-handlebody}.
It is clear from the construction that the dynamical handlebody
underlying $(\wt W,\wt\alpha)$ is isomorphic to $W$.

\begin{figure}[h]
\centering
\includegraphics[width=0.7\linewidth]{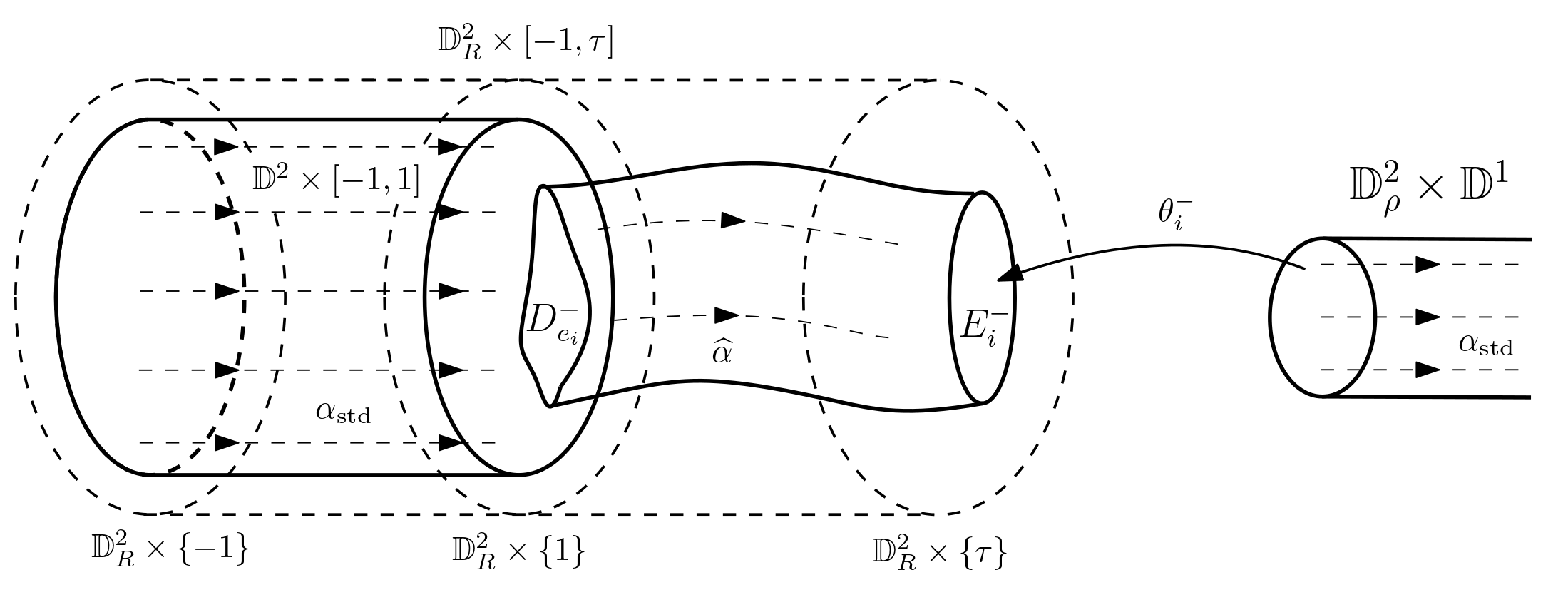}
\caption{Subhandlebody $(\wt W, \wt \alpha)$ inside $(\wh W, \wh \alpha)$}
\label{fig:subhandlebody}
\end{figure}

\subsection{Embedding a Reeb handlebody into a contact manifold}\label{ss:emb-Reeb}

The following lemma shows that each Reeb handlebody can be realized
inside an arbitrary given contact $3$-manifold.

\begin{lem}\label{lem:Reeb-handlebody-emb}
Let $(W,\alpha)$ be a Reeb handlebody, and $(M,\xi)$ any cooriented contact
$3$-manifold. Then every smooth embedding $\phi_0:W\into M$ is 
isotopic to an embedding $\phi:W\into M$ such that
$\phi_*\ker\alpha=\xi$. 
\end{lem}

\begin{rem}\label{rem:Reeb-handlebody-emb}
(a) The dynamical handlebody $W$ retracts onto an embedding of its underlying
graph $G$ such that the vertices are the centers of the corresponding
$0$-handles and the edges run through the corresponding $1$-handles. 
This induces a one-to-one correspondence between isotopy classes of
embeddings $W\into M$ and isotopy classes of embeddings $G\into M$.
Thus Lemma~\ref{lem:Reeb-handlebody-emb} asserts the existence of a
contact embedding $\phi$ realizing a prescribed embedding of the graph
$G$, up to isotopy that can be fixed at the vertices.
Note that we can embed the graph $G$ into any $3$-manifold with
  arbitrary prescribed isotopy classes of the edges.

(b) We can extend the pushforward $\phi_*\alpha$ on $\phi(W)$ to a
defining contact form $\beta$ for $\xi$ on $M$, so that
$\phi^*\beta=\alpha$.
\end{rem}

\begin{proof}
By Corollary~\ref{cor:contact-tubes}, we may assume that $(W,\alpha)$ is given by $(\wt W,\wt\alpha)\subset(\wh W,\wh\alpha)$, as
obtained from the alternative construction in~\S\ref{ss:dyn-Reeb},
with constants $R\geq 1$ and $\tau>1$.

According to Remark~\ref{rem:Reeb-handlebody-emb}(a), the embedding
$\phi_0$ yields an embedding of the underlying graph $G$ into $M$,
consisting of distinct points $a_q$ for the vertices $q\in Q$ and an
embedded path $\delta_e$ from $a_q$ to $a_{q'}$ for each edge $e$ from
$q$ to $q'$. After an isotopy, we may assume that each $\delta_e$ is
transverse to $\xi$.   
By the contact neighbourhood theorem (see e.g.~\cite[Theorem
  2.5.15]{geiges2008introduction}), an open neighbourhood $(U_e,\xi)$
of $\delta_e$ is contactomorphic to a neighbourhood of a vertical path
in $(\R^3,\xi_\std)$, in particular it is tight. 

For each vertex $q$, let $U_q$ be the intersection of all the sets
$U_e$ for edges $e$ adjacent to $q$. Take a Darboux chart
$(\wt B_q:=\DD^2_R\times[-\tau-1,\tau+1],\xi_\std)\subset (U_q,\xi)$. 
For each edge $e$ from $q$ to $q'$ consider the images
$E_e^\pm=\theta_e^\pm(\DD^2_\rho)\subset\DD^2_R$ under the attaching maps,
which are ellipsoids with center $z_e^\pm=(x_e^\pm,0)$. 
Pick a transverse path $\gamma_e\subset U_e$ from $(z_e^-,\tau)\in\wt B_q$ to
$(z_e^+,-\tau)\in\wt B_{q'}$ such that
$\gamma_e\cap\wt B_q = \{z_e^-\}\times[\tau,\tau+1]$ and   
$\gamma_e\cap\wt B_{q'} = \{z_e^+\}\times[-\tau-1,-\tau]$.
Moreover, we can arrange that the composition of $\gamma_e$ with
  the straight lines to $a_q,a_q'$ in $\wt B_q,\wt B_{q'}$ is smoothly
  isotopic to $\delta_e$.
Thicken $\gamma_e$ to an embedding
$\phi_e:H_e=\DD^2_\rho\times\DD^1\into U_e$ whose restriction to the
vertical boundary $\p^vH_e$ is transverse to $\xi$, and such that
$\phi_e(z,t)=(\theta_e^-(z),t+\tau+1)$ near $t=-1$ and 
$\phi_e(z,t)=(\theta_e^+(z),t-\tau-1)$ near $t=1$, in the Darboux charts. 
We perform this for all edges $e$ such that the images of the $\phi_e$ 
are disjoint. 
The maps $\phi_e$ together with the inclusions $\wh
B_q:=\DD^2_R\times[-\tau,\tau]\into U_q$ define an embedding
$\wh\phi:\wh W\into M$. 
By construction, the pullback $\wh\phi^*\xi$ agrees with $\xi_\std$ on
the $0$-handles $\wh B_q$ and near $\p^\pm H_e$ for all $1$-handles.   
Using Lemma~\ref{lem:contact-tubes}, we can therefore adjust $\wh\phi$
on the $1$-handles so that $\wh\phi^*\xi=\xi_\std$ everywhere. 
The restriction of $\wh\phi$ to $\wt W\subset\wh W$ is the desired
embedding $\phi:W=\wt W\into M$ with $\phi_*\ker\alpha=\xi$. 
Since the embedding $\phi:W\into M$ realizes the given isotopy class
of the embedding $G\into M$, it is isotopic to $\phi_0$ by
Remark~\ref{rem:Reeb-handlebody-emb}(a).
\end{proof}

\subsection{Undecidable recurrence and periodic Reeb orbits}

We finish by outlining two immediate consequences of the computable return-map
construction of Theorem \ref{thm:main}.

The first application concerns recurrence to a fixed transverse
section and its decidability. Recall that a set $A \subset \NN$ is said to be \emph{decidable} (also known as \emph{recursive}) if there exists a computable total function $f: \NN \to \NN$ such that $f(n) = 1$ if $n \in A$ and $f(n) = 0$ if $n \not\in A$. Otherwise, the set $A$ is said to be \emph{undecidable}.

\begin{cor}\label{cor:undecidable-recurrence}
Let $(M,\xi)$ be a coorientable contact $3$-manifold. There exists a
contact form $\beta \in \Omega^1(M)$ with $\xi = \ker\beta$ and a disk
$D\subset M$ transverse to the Reeb flow of $\beta$ such that the set
$$
    A = \left\{x \in \NN \subset D \,\mid\, \text{the forward Reeb orbit of $x$ returns to $D$}\right\}
$$
is undecidable.
\end{cor}

\begin{proof}
Choose a partial computable function
$f:\mathbb N\dashrightarrow \mathbb N$ with undecidable domain, for
instance the halting function of a universal Turing machine.  By
Theorem~\ref{thm:main}, the contact form and section $D \subset M$ can be chosen so
that the first return map $\Phi:D\dashrightarrow D$ realizes $f$ on
the encoded copy of $\mathbb N \subset D$.  Thus $\Phi(x)$ is defined precisely
for those inputs $x$ belonging to the domain of $f$.  Deciding return
to $D$ would therefore decide the domain of $f$, a contradiction.
\end{proof}

The Reeb realization of generalized shifts also provides a direct source of
infinitely many simple periodic Reeb orbits. It is enough to realize the elementary generalized shift given by the Bernoulli shift.

Recall that the (right) shift map, also known as the \emph{Bernoulli shift}, is the map
\[
  s_+: \{0,1\}^\ZZ\longrightarrow \{0,1\}^\ZZ,
  \qquad
  s_+((t_i)_{i\in\mathbb Z})=(t_{i+1})_{i\in\mathbb Z}.
\]
Notice that this map is an extension to non-finite sequences of the shift map $s_+: \Lambda \to \Lambda$ used in Section \ref{ss:basic-ops}. In particular, since the embedding $\kappa_\Lambda: \Lambda \hookrightarrow C^2$ can be extended to a homeomorphism $\kappa: \{0,1\}^{\mathbb Z} \stackrel{\cong}{\to} C^2$, the shift map also gives rise to a Bernoulli shift on the square Cantor set, also denoted $s_+: C^2 \to C^2$.

\begin{cor}\label{cor:infinitely-many-reeb-orbits}
Let $(M,\xi)$ be a coorientable contact $3$-manifold. Then there exists
a contact form $\beta$ with $\xi = \ker\beta$ whose Reeb flow has infinitely many distinct closed orbits. More precisely, $\beta$ may be
chosen so that, for a transverse disk $D$, the first return map contains an
invariant square Cantor subset on which the dynamics is conjugate to the Bernoulli shift.
\end{cor}

\begin{proof}
The function $s_+: \Lambda \to \Lambda$ is computable, so by Theorem \ref{thm:main}, there exists a contact form and a section $D \subset M$ such that the first return map $\Phi: D \dashrightarrow D$ of the Reeb flow satisfies $\Phi|_{C^2} = s_+$.

Now, observe that the Bernoulli shift $s_+: C^2 \to C^2$ has periodic points of arbitrarily large period. Indeed, for each $q\geq 1$, consider the two-sided sequence
$t = (t_i)_{i\in\mathbb Z}\in\{0,1\}^{\mathbb Z}$ defined periodically by
\[
  t_i=1 \quad \text{if } i\equiv 0 \pmod q,
  \qquad
  t_i=0 \quad \text{otherwise}.
\]
This sequence satisfies $s_+^k(t) \neq t$ for $1 \leq k < q$ and $s_+^q(t) = t$, so it exactly has period $q$ for the shift $s_+$. %
Each such periodic point suspends to a
closed Reeb orbit on $M$. Furthermore, if the period is $q$, then the corresponding
suspended orbit meets the transverse disk exactly $q$ times before closing;
hence it cannot be an iterate of an orbit with smaller return period.
Since the primitive periods are unbounded, the corresponding simple closed
Reeb orbits are distinct, proving the claim.
\end{proof}

\begin{rem}
The  former result should be compared with the two-or-infinity periodic
orbits theory for Reeb flows in dimension three, initiated by
Hofer--Wysocki--Zehnder \cite{HoferWysockiZehnder1998} and developed
further by Cristofaro-Gardiner--Hutchings--Pomerleano
\cite{CristofaroGardinerHutchingsPomerleano2019} and
Cristofaro-Gardiner--Hryniewicz--Hutchings--Liu
\cite{CristofaroGardinerHryniewiczHutchingsLiu2023}.  Under the
hypotheses of these results, a Reeb flow has either two or infinitely
many simple periodic orbits.  The exceptional case of exactly two simple
orbits is highly rigid: it can occur only on \(S^3\) or on a lens space,
where the two orbits are the core circles of a genus-one Heegaard
splitting \cite{CristofaroGardinerHryniewiczHutchingsLiu2023b}.

Thus Corollary~\ref{cor:infinitely-many-reeb-orbits} is not meant as a
new abstract existence theorem in the range already covered by this
theory.  Its point is different: it gives a prescribed contact form whose
infinitely many simple Reeb orbits arise from an explicit computational
mechanism.  The two-symbol subsystem produces primitive periodic points
of unbounded return period, and their suspensions rule out, by
construction, the exceptional two-orbit dynamics.
\end{rem}

\bibliographystyle{abbrv} %
\bibliography{bibliography}

\end{document}